\documentclass{article}
\usepackage{amsmath, amssymb, latexsym, amscd, amsthm,amsfonts,amstext}
\usepackage[mathscr]{eucal}
\usepackage{graphics, graphpap}
\usepackage{makeidx}
\usepackage{xspace}
\usepackage{array, tabularx, longtable}
\usepackage{multicol,color}
\usepackage{color}

\def\RR{{\mathbb R}}

\def\liml{\lim\limits}
\def\suml{\sum\limits}

\newtheorem{theorem}{\bf Theorem}[section]

\newtheorem*{theorem*}{Theorem}

\usepackage{amsmath, amssymb, latexsym, amscd, amsthm,amsfonts,amstext}
\usepackage[mathscr]{eucal}
\usepackage{graphics, graphpap}
\usepackage{makeidx}
\usepackage{xspace}
\usepackage{array, tabularx, longtable}
\usepackage{multicol,color}

\def\suml{\sum\limits}

\def\({\left(}
\def\){\right)}

\def\suml{\sum\limits}

\def\liml{\lim\limits}

\begin{document}
\title{A note on a Sung-Wang's paper}
\author{Nguyen Thac Dung}

\date{\today}
\maketitle

\begin{abstract}
The purpose of this note is to study the connectedness at infinity of manifold by using the theory of $p$-harmonic functions. We show that if the first eigenvalue $\lambda_{1,p}$ for the $p$-Laplacian achievies its maximal value on a K\"{a}hler manifold or a quaternionic K\"{a}hler manifold then such a manifold must be connected at infinity unless it is a topological cylinder with an explicit warped product metric.  
\noindent 
\vskip0.2cm



\end{abstract}
\large
\vskip0.4cm

\section{Introduction}
It is well-known that the theory of $L^2$ harmonic functions/forms has close relation to geometry of manifolds, in particular, geometric structure at infinity. We refer the reader to \cite{KLZ, li, LW1, LW2, LW3} for further details of this topic. From a variational point of view, $p$-harmonic functions are natural extensions of harmonic functions. Therefore, it is very natural to study $p$-harmonic functions/forms on submanifolds and ask what is the ralationship between geometry of such these submanifolds and the space of $p$-harmonic functions/forms. We emphasize that compared with the theory for harmonic functions, the study of $p$-harmonic functions is generally harder, even though elliptic, is degenerate and the regularity results are far weaker. In \cite{BK}, Buckuley and Koskela gave volume estimate of $p$-parabolic ends, $p$-nonparabolic ends in term of the first eigenvalue of the $p$-Laplacian. Then, in \cite{BCS}, Batista, Cavalcante and Santos used $p$-harmonic function to introduce a definition of $p$-parabolic ends (also see \cite{PST}). They proved that  if $E$ is an end of a complete Riemannian manifold and satisfies a Sobolev-type inequality then $E$ must either have finite volume or to be $p$-nonparabolic. A characterization of $p$-nonparabolic ends in the context of submanifold is also verified. Recently, Chang, Chen and Wei (see \cite{CCW}) studied $p$-harmonic maps with finite $q$-energy and prove a Liouville type theorem. As an application, they extended this theorem to some $p$-harmonic maps such as $p$-harmonic morphisms and conformal maps between Riemannian manifolds. 

The main purpose of this note is to understand manifolds whose principal eigenvalue $\lambda_{1,p}$ for the $p$-Laplacian achieves its maximal. When $p=2$, this problem has been studied by Li and Wang in \cite{LW1, LW2, LW3}. In \cite{KLZ}, Kong, Li and Zhou proved a splitting theorem on quaternionic K\"{a}hler manifolds. In the general case of $p\geq2$, Sung and Wang generalized Li-Wang's results on a Riemannian manifold with $\lambda_{1,p}$ obtaining its maximal value. They showed that such a manifold must be connected at infinity unless it is a topological cylinder endowed with an explicit warped product metric. Motivated by these results, we consider the same problem on K\"{a}hler manifolds and obtain following two theorems.
\begin{theorem}\label{main3}
Let $M^{2m}$ be a complete K\"{a}hler manifold of complex dimension $m\geq1$ with holomorphic bisectional curvature bounded by 
$$ BK_M\geq-1. $$
If $\lambda_{1,p}\geq\left(\frac{2m}{p}\right)^p$, then either
\begin{enumerate}
\item $M$ has no $p$-parabolic end; or 
\item $M$ splits as a warped product $M=\RR\times N$ where $N$ is a compact manifold. Moreover, the metric is given by
$$ ds_M^2=dt^2+e^{-4t}\omega_2^2+e^{-2t}\suml_{\alpha=3}^{2m}\omega_{\alpha}^2, $$
where $\left\{\omega_2,\ldots, \omega_{2m}\right\}$ are orthonormal coframes for $N$.
\end{enumerate}
\end{theorem}
Similarly, we obtain a splitting theorem on quaternionic K\"{a}hler manifolds, under a weaker assumption on the scalar curvature
\begin{theorem}\label{main4}
Let $M^{4m}$ be a complete noncompact quaternionic K\"{a}hler manifold of real dimension $4m$ with the scalar curvature of $M$ bounded by
$$ S_M\geq-16m(m+2). $$
If $\lambda_{1,p}\geq\left(\frac{2(2m+1)}{p}\right)^p$, then either
\begin{enumerate}
\item $M$ has no $p$-parabolic end; or 
\item $M$ splits as a warped product $M=\RR\times N$ where $N$ is a compact manifold. Moreover, the metric is given by
$$ ds_M^2=dt^2+e^{4t}\suml_{p=2}^4\omega_p^2+e^{2t}\suml_{\alpha=5}^{4m}\omega_{\alpha}^2, $$
where $\left\{\omega_2,\ldots, \omega_{4m}\right\}$ are orthonormal coframes for $N$.
\end{enumerate}
\end{theorem}
The note has two sections. In the section 2, we give a unified proof of Theorems \ref{main3} and \ref{main4}. 
\section{Structure theorems on K\"{a}hler manifolds with maximal $\lambda_{1,p}$}
\setcounter{equation}{0}
In this section, we provide a unified proof of Theorems \ref{main3} and \ref{main4}. Our argument is close to the proof of Theorem 3.1 in \cite{SW}. First, recall that a smooth function $u$ is said to be $p$-harmonic if
$$ \Delta_pu:=div(|\nabla u|^{p-2}\nabla u)=0. $$
\begin{proof}[Proof of theorems \ref{main3} and \ref{main4}]
Note that by theorem 5.1 and 5.2 in \cite{BCS}, we have that $\lambda_{1,p}$ achieves its maximal value. Suppose that $M$ has a $p$-parabolic end $E$. Let $\beta$ be the Busemann function associated with a geodesic ray $\gamma$ contained in $E$, namely,
$$ \beta(q)= \liml_{t\to\infty}\left(t-dist(q, \gamma(t))\right).$$
The Laplacian comparison theorems in \cite{LW3} and \cite{KLZ} imply
$$ \Delta\beta\geq -a, $$
where $a=2m$ in theorem \ref{main3} (see \cite{LW3}) and $a=2(2m+1)$ in theorem \ref{main4} (see \cite{KLZ}). Hence, for $b=\frac{a}{p}$, we have
$$ \begin{aligned}
\Delta_p(e^{b\beta})&=div\left(b^{p-2}e^{b(p-2)\beta}\nabla(e^{b\beta})\right)\\
&=b^{p-1}div\left(e^{b(p-1)\beta}\nabla\beta\right)\\
&=b^{p-1}\left(e^{b(p-1)\beta}\Delta\beta+b(p-1)e^{b(p-1)\beta}|\nabla\beta|^2\right)\\
&\geq b^{p-1}e^{b(p-1)\beta}\left(-bp+b(p-1)\right)=-b^p\left(e^{b\beta}\right)^{p-1}.
\end{aligned} $$
Let $g=e^{b\beta}$, we obtain
$$ \Delta_p(g)\geq-\lambda_{1,p}g^{p-1}. $$
The variational characterization of $\lambda_{1,p}$ gives
$$ \lambda_{1,p}\int_M(\phi g)^p \leq \int_M|\nabla(\phi g)|^p,$$
for any nonnegative compactly supported smooth function $\phi$ on $M$. Integration by parts implies
$$ \int_M\phi^pg\Delta_pg=-\int_M\phi^p|\nabla g|^p-p\int_M\phi^{p-1}g\left\langle\nabla\phi, \nabla g\right\rangle|\nabla g|^{p-2}.  $$
We note that
$$ \begin{aligned}
|\nabla(\phi g)|^p
&=\left(|\nabla\phi|^2g^2+2\phi g\left\langle\nabla\phi, \nabla g\right\rangle+\phi^2|\nabla g|^2 \right)^\frac{p}{2}\\
&\leq \phi^p|\nabla g|^p+p\phi g\left\langle\nabla\phi, \nabla g\right\rangle\phi^{p-2}|\nabla g|^{p-2}+c|\nabla\phi|^2g^{p}, 
\end{aligned} $$
for some constant $c$ depending on $p$. Therefore, we have
 \begin{align}
\int_M&\phi^pg(\Delta_p(g)+\lambda_{1,p}g^{p-1})\notag\\
&=\lambda_{1,p}\int_M(\phi g)^{p}-\int_M\phi^p|\nabla g|^p-p\int_M\phi^{p-1}g\left\langle \nabla\phi, \nabla h\right\rangle|\nabla h|^{p-2}\notag\\
&\leq \int_M|\nabla(\phi g)|^p-\int_M\phi^p|\nabla g|^p- p\int_M\phi^{p-1}g\left\langle \nabla\phi, \nabla h\right\rangle|\nabla h|^{p-2}\notag\\
&\leq c\int_M|\nabla\phi|^2g^{p}.\label{es31}
\end{align}
Now, for $R>0$, we choose the test function $0\leq\phi\leq 1$ such that 
$$ \phi=\begin{cases}
1,\quad&\text{ on }B(R)\\
0&\text{ on }M\setminus B(2R)
\end{cases} $$ 
and $|\nabla\phi|\leq\frac{2}{R}$. It turns out that
\begin{align}
\int_M|\nabla\phi|^2g^p
&=\int_M|\nabla\phi|^2e^{a\beta}\leq\frac{4}{R^2}\int_{B(2R)\setminus B(R)}e^{a\beta}\notag\\
&\leq\frac{4}{R^2}\int_{E\cap(B(2R)\setminus B(R))}e^{a\beta}+\frac{4}{R^2}\int_{(M\setminus E)\cap(B(2R)\setminus B(R))}e^{a\beta}.\label{es32}
\end{align}
Since $\lambda_{1,p}=b^p=\left(\frac{a}{p}\right)^p$, the theorem 0.1 in \cite{BK} implies
$V(E\setminus B(R))\leq ce^{-aR}.$
Therefore, the first term of \eqref{es32} tends to $0$ when $R\to\infty$. Note that Li and Wang (see \cite{LW3}) showed that 
$$ \beta(q)\leq-r(q)+c $$
on $M\setminus E$. Moreover, by volume comparison theorem in \cite{KLZ, LW3}, we have $V(B(R))\leq ce^{aR}$. This implies that the second term of \eqref{es32} converges to $0$ as $R$ goes to infinity. Hence, by \eqref{es31} we have 
$$ \Delta_pg+\lambda_{1,p}g^{p-1}\equiv0. $$
Thus,
$$ \Delta\beta=-a. $$
The conclusion is followed by using the argument in \cite{LW3} (in theorem \ref{main3}) and \cite{KLZ} (in theorem \ref{main4}). The proof is complete.
\end{proof}

\section*{Acknowledgment:} The author was supported in part by NAFOSTED under grant number 101.02-2014.49. A part of this note was written during a his stay at Vietnam Institute for Advance Study in Mathematics (VIASM). He would to express his sincerely thanks to staffs there for excellent working condition and financial support.

\addcontentsline{toc}{section}{3 \hspace{0.09cm} References}

\vskip0.4cm

\bigskip
\noindent
Nguyen Thac Dung \\
Department of Mathematics, Mechanics, and Informatics (MIM)\\
Hanoi University of Sciences (HUS-VNU)\\
Vietnam National University\\
334 Nguyen Trai Str., Thanh Xuan, Hanoi\\
{\tt E-mail:dungmath@yahoo.co.uk}

\vskip0.4cm

\end{document}